\documentclass[a4paper]{article}
\usepackage[utf8]{inputenc}
\usepackage[T1]{fontenc}
\usepackage{graphicx}

\usepackage{amsmath}

\usepackage{amsthm}

\usepackage{amssymb}
\usepackage{mathabx}

\usepackage[all]{xy}

\usepackage[english]{babel}

\DeclareMathSymbol{\theta}{\mathord}{letters}{"23}
\DeclareMathSymbol{\vartheta}{\mathord}{letters}{"12}
\DeclareMathSymbol{\phi}{\mathord}{letters}{"27}
\DeclareMathSymbol{\varphi}{\mathord}{letters}{"1E}

\newcommand\CC{\mathbb{C}}

\newcommand\ZZ{\mathbb{Z}}
\newcommand\NN{\mathbb{N}}

\newcommand\ra{\rightarrow\nobreak}

\newcommand\ab{\hbox{$(a,b)$-mod}\-ule}
\newcommand\CCb{\mathbb{C}[[b]]}
\newcommand\Ead{\widecheck{E}^*} 

\DeclareMathOperator{\diff}{d}

\DeclareMathOperator{\df}{df}
\DeclareMathOperator{\Ker}{Ker}

\DeclareMathOperator{\Hom}{Hom}

\newtheorem{theorem}{Theorem}

\newtheorem{lemma}[theorem]{Lemma}
\newtheorem{definition}[theorem]{Definition}

\theoremstyle{definition}

\theoremstyle{remark}
\newtheorem{remark}[theorem]{Remark}

\numberwithin{theorem}{section}

\title{Jordan-Hölder decomposition of regular hermitian {\ab}s}
\author{Piotr P. Karwasz
\thanks{Institut Matematyki,
  Uniwersytet Gdański, ul. Wita Stwosza 57, 80-952 Gdańsk, Polska\newline
  E-mail: \texttt{piotr.karwasz@mat.ug.edu.pl}\newline
  partially funded by grant NCN 2013/10/E/ST1/00688 ``Rozmaitości o trywialnej pierwszej klasie Cherna''}}
\begin{document}
\maketitle

\begin{abstract}

A classical result of singularity theory states that the spectrum of an isolated hypersurface singularity is symmetric
with respect to $n/2$, where $n$ is the dimension of the enclosing space. We prove a similar result for the
Jordan-Hölder composition series of the {\ab} associated to an isolated hypersurface singularity.

\end{abstract}
\section{Introduction}

Let $f:(\CC^{n+1}, 0)\ra(\CC, 0)$ be a germ of holomorphic function with an isolated singularity at the origin.
The lattice introduced by \textsc{E.~Brieskorn} in \cite{brieskorn}
\[
  D := \frac{\Omega_0^{n+1}}{\df\wedge\Omega_0^{n-1}},
\]
where $\Omega_0^{p}$ are germs of holomorphic $p$-forms at the origin, is endowed with several structures: a structure as free
$\CC\{a\}$-module and $\CC\{\{b\}\}$-module, where $a:=f\cdot$ and $b:= \df \wedge \diff^{-1}$, a V-filtration and two
different mixed Hodge structures defined by Varchenko (\cite{varchenko}) and Steenbrink (\cite{steenbrink}).

Many invariants of the hypersurface singularity, such as the spectrum and complex monodromy (cf. \cite{schulze2}) can be
computed using the $\CC\{a\}$-module and $\CC\{\{b\}\}$-module structure. Following the work of \textsc{D.~Barlet} we are
therefore interested in the properties of the $b$-adic completion of the Brieskorn lattice considered as an abstract algebraic structure
called {\ab} (\cite{abmodules}):
\begin{definition}
An {\ab} is a free $\CCb$-module $E$ of finite rank over the ring of formal power series in $b$, endowed with a
$\CC$-linear endomorphism `$a$' which satisfies:
  \[ ab - ba = b^2 \]
\end{definition}
The simplest examples of {\ab}s have rank $1$ and are generated by an element $e_\lambda$ satisfying $ae_\lambda =
\lambda be_\lambda$ for a complex number $\lambda$. We will refer to them as \textbf{elementary} {\ab}s of parameter
$\lambda$ and note them with $E_\lambda$.

Since all {\ab}s $E$ coming from the geometric setting are \textbf{regular}, i.e. they are sub-{\ab}s of an {\ab} $E^\#$ which satisfies
$aE^\# \subset bE^\#$, we will be treating only the case of regular {\ab}s.

For an {\ab} $E$ let us consider a filtration:
\[
0 = F_0 \subsetneq F_1 \subsetneq \ldots \subsetneq F_n = E
\]
where the $F_i$ are sub-{\ab} and are \textbf{normal}, i.e. such that $E/F_i$ is a free $\CC[[b]]$-module and hence is still an {\ab}.
We call such filtrations \textbf{Jordan-Hölder} composition series. The main difference between these filtrations and the Hodge and V-filtration
on an {\ab} is that the latter are not normal in general.

The higher residue pairings defined by \textsc{K.~Saito} (\cite{saito}) on the Brieskorn lattice may be viewed as a polarisation
of the mixed Hodge structure on the Brieskorn lattice (\cite{hertlingbl}. This polarisation induces on the associated
{\ab} a non-degenerate hermitian form (cf. \cite{karwasz}).

\textsc{R.~Belgrade} (\cite{belgrade}) uses this form to produce a purely algebraic proof of the symmetry of the spectral numbers
of the isolated hypersurface singularity. Since for regular {\ab} all the quotients $F_j/F_{j-1}$ of a Jordan-Hölder composition series are elementary 
{\ab}s $E_{\lambda_i}$, $\lambda_i\in\CC$, a question may arise, whether these $\lambda_i$ have some symmetry properties.

Unfortunately the quotients of a Jordan-Hölder composition series are not unique up to permutation and vary between series (\cite{abmodules}),
therefore the answer is negative for a general composition series. We may however provide a result similar to the symmetry of a spectrum for a
particular Jordan-Hölder series::
\begin{theorem}
\label{big}
Let $E$ be a regular {\ab} with a non-degenerate hermitian form then it has at least one Jordan-Hölder composition series
\[
0=F_0\subsetneq F_1 \subsetneq \ldots \subsetneq F_n = E
\]
satisfying the following symmetry: the quotients $F_{n-i}/F_{n-i-1}$ and $F_i/F_{i-1}$ are adjoint one of the other and the hermitian form
induces a non-degenerate hermitian form on $F_{n-i}/F_i$.
\end{theorem}

\section{Hermitian form on {\ab}s}

The translation of \textsc{K.~Saito}'s higher residue pairings into the language of {\ab}s suggests the following
pair of definitions:
\begin{definition}
Let $E$ be an {\ab}-module. We call \textbf{dual} {\ab} and denote it by $E^*$, the $\CC[[b]]$-module $\Hom_{\CC[[b]]}(E,
E_0)$ endowed with the $a$-action:
\[
[a\cdot\phi](x) = a\phi(x) - \phi(ax).
\]
where $\phi\in E^*$ and $x\in E$.
\end{definition}
\begin{definition}
Let $E$ be an {\ab}-module. We call \textbf{conjugate} {\ab} and denote it by $\widecheck{E}$, the $\CC$-vector space $E$
endowed with the action:
\begin{align*}
a\cdot_{\widecheck{E}}x &= -a\cdot_E x\\
b\cdot_{\widecheck{E}}x &= -b\cdot_E x
\end{align*}
where $x\in E$.
\end{definition}
The two operations of taking the dual and the conjugate are functors (contravariant and covariant) from the category of
{\ab}s into itself. The composition of both functors is a contravariant functor called \textbf{adjonction}
functor.Keeping this in mind Saito's higher residue pairings are equivalent to the existence of an isomorphism
$\Phi:E\ra\widecheck{E}^*$ of {\ab}s such that $\widecheck{\Phi}^* = \Phi$ (cf. \cite{karwasz}). We will call such an
isomorphism a \textbf{hermitian structure} by analogy to the category of complex vector spaces.

If $\Phi$ is a hermitian structure, then we can associate it a map:
\begin{align*}
H: E \times E &\ra E_0\\
(x,y) &\mapsto \Phi(y)(x)
\end{align*}
satisfying:
\begin{align*}
H(bx, y) &= bH(x, y) = H(x, -by)\\
aH(x, y) &= H(ax, y) - H(x, ay)\\
H(x, y) &= S(b)e_0 \Rightarrow H(y, x) = S(-b)e_0
\end{align*}
that we call \textbf{hermitian form} associated to $\Phi$.

We should remark that not every {\ab} admits a hermitian structure. For most {\ab}s $E$ and $\widecheck{E}^*$ are not
isomorphic and even if they are there are examples of {\ab}s that only admit an anti-hermitian ($\widecheck{\Phi}^* =
-\Phi$) structure. Therefore we call \textbf{self-adjoint} the {\ab}s satisfying $E\simeq\widecheck{E}^*$ and
\textbf{hermitian} those admitting a hermitian structure.

In the general case (\cite{karwasz}) a regular self-adjoint {\ab} can be decomposed (not necessarily in a unique way)
into the direct sum of a hermitian {\ab} and an anti-hermitian one.

Let $\{F_i\}$ be a composition series of a regular hermitian {\ab} $E$:
\[
0=F_0\subsetneq F_1\subsetneq \ldots \subsetneq F_n = E
\]
and let $G_i$be the subset of $\widecheck{E}^*$ which is zero on $F_i$. In this way we obtain a
composition series of $\widecheck{E}^*$:
\[
0=G_n \subsetneq G_{n-1} \subsetneq \ldots \subsetneq G_0 = \widecheck{E}^*
\]
and we may provide the following definition:
\begin{definition}
Let $\{F_i\}$ be a composition series of a regular hermitian {\ab} $E$ with a hermitian structure
$\Phi:E\ra\widecheck{E}^*$. We say that $\{F_i\}$ is \textbf{self-adjoint} if:
\[
\Phi(F_i) = G_{n-1}
\]
or equivalently $\widecheck{F_i/F_{i-1}} \simeq F_{n-i+1}/F_{n-i}$ and $\Phi$ induces a hermitian structure on
$F_{n-i}/F_i$. 
\end{definition}

\section{Proof of theorem \ref{big}}

Before proving the theorem we shall introduce a couple of lemmas.

\begin{lemma}
\label{delta=0}
Let $E$ be a regular hermitian {\ab}, $\Phi:E\ra \Ead$ a hermitian structure and $H$ the associated hermitian form.

If there exists a normal sub-{\ab} $F_1$ of rank $1$ such that \[H(F_1,F_1)=0,\] then there exists a normal sub-{\ab}
$F_{n-1}$ of rank $n-1$ such that $E/F_{n-1} \simeq \widecheck{F_1}^*$ and $F_{n-1}/F_1$ is hermitian.
\end{lemma}

\begin{proof}
Let $F_1 \simeq E_\lambda$ and $e_\lambda$ be the generator of $F_1$ and consider the annihilator of this form under
$H$:
\[
F_{n-1} := \{x\in E | H(e_\lambda,x) = 0 \}.
\]

We remark that the condition $H(e_\lambda,e_\lambda)=0$ gives us $F_1 \subset F_{n-1}$ and $F_{n-1}$ is normal, because
it is the kernel of a morphism.

Let us consider the following exact sequence:
\[
0 \ra F_1 \ra E \ra E/F_1 \ra 0
\]
from which we can pass to the adjoint sequence:
\[
0 \ra \widecheck{\left(E/F_1\right)^*} \ra \Ead \overset{\pi}{\ra} \widecheck{F}_1^*
\ra 0.
\]
Since $\pi$ is the restriction morphism of forms on $E$ to the sub-{\ab} $F_1$, the kernel of $\pi$, $K:=\Ker\pi$ can be
described as follows:
\[
K=\{\phi \in \Ead | \phi(F_1)=0\}.
\]
The adjoint sequence being exact, we can identify from now on $\widecheck{\left(E/F_1\right)}^*$ with $K$, i.e.
sub-{\ab} of $\Ead$ whose elements annihilate $F_1$.

If we consider the restriction of the map $\Phi$ to $F_{n-1}$
\[
\Phi|_{F_{n-1}}: F_{n-1} \ra \Ead
\]
and the fact that by definition $\Phi(x)(e_\lambda)=0$ for all $x\in F_{n-1}$, we obtain that $\Phi(F_{n-1})\subset
\widecheck{\left(E/F_1\right)^*}$.

On the other side for all $\phi\in \widecheck{\left(E/F_1\right)^*}$ the element $y=\Phi^{-1}(\phi)$ verifies
$\Phi(y)(e_\lambda)=0$, therefore we have also $\widecheck{\left(E/F_1\right)^*}\subset \Phi(F_{n-1})$. It follows that
$\Phi(F_{n-1})=\widecheck{\left(E/F_1\right)^*}$ and since $\Phi$ is an isomorphism, $F_{n-1}$ is isomorphic to its
image by $\Phi$: $\widecheck{\left(E/F_1\right)^*}$.

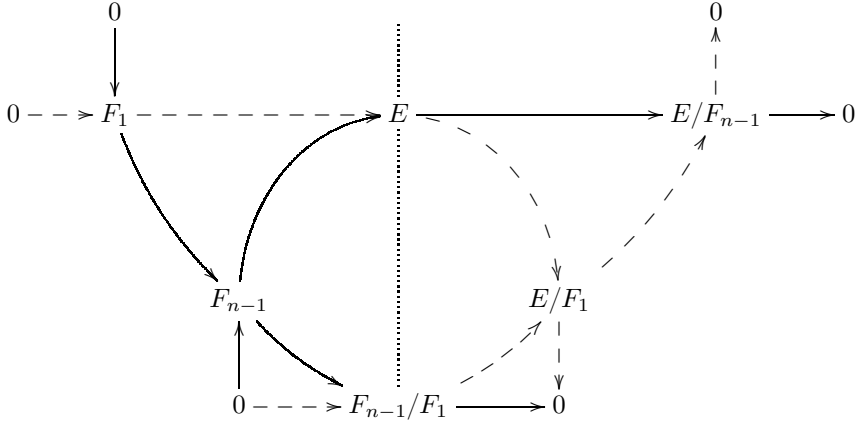
\begin{figure}
\caption{Modules in symmetric positions with respect to the dotted line are each
other's adjoint.}
  \label{fig:exact_sequences}
\[
\xymatrix{
 & 0 \ar[d] & & \ar@{.}[d] & & 0 &\\
0 \ar@{-->}[r] & F_1 \ar@{-->}[rr] \ar@/_/[ddr] & & E \ar[rr]
\ar@{-->}@(r,u)[ddr] \ar@{.}[ddd] & &
E/F_{n-1} \ar@{-->}[u] \ar[r] & 0 \\
 & & & & & & \\
& & F_{n-1} \ar@(u,l)[uur] \ar@/_/[dr] & & E/F_1 \ar@{-->}[d]
\ar@/_/@{-->}[uur]& & \\
 & & 0 \ar[u] \ar@{-->}[r] & F_{n-1}/F_1 \ar@/_/@{-->}[ur] \ar[r] &  0 & & \\
}
\]
\end{figure}

Let us look now at the following exact sequence:
\[
0\ra (F_{n-1}/F_1) \ra (E/F_1) \ra (E/F_{n-1}) \ra 0
\]
and its adjoint sequence:
\[
0\ra \widecheck{\left(E/F_{n-1}\right)^*} \overset{i}{\ra}
\widecheck{\left(E/F_1\right)^*} \overset{\pi}{\ra}
\widecheck{\left(F_{n-1}/F_1\right)^*} \ra 0.
\]
$\pi$ designates the restriction application on the forms of
$\widecheck{\left(E/F_1\right)^*}$. $\Ker\pi$ is thus the
forms of $\widecheck{\left(E/F_1\right)^*}$ that annihilate
$\widecheck{\left(F_{n-1}/F_1\right)^*}$ or with the
convention of the previous paragraph, the forms of $\widecheck{E}^*$ that annihilate
$F_{n-1}$ and $F_1\subset F_{n-1}$:
\[
\Ker\pi = \{\phi\in \widecheck{E}^* \textrm{ s.t. } \phi(F_{n-1})=0\}
\]
We note that the hermitianity of $\Phi$ gives us
\[
\Phi(e_\lambda)(F_{n-1}) = \widecheck{\Phi(F_{n-1})(e_\lambda)} = 0
\]
and therefore we have $\Phi(F_1)\subset \Ker\pi$. An easy calculation shows that
$\Ker\pi$ is of rank $1$. Since
$\Phi(F_1)$ is normal, of rank~$1$ and included into $\Ker\pi$, they must be
equal.

We obtain $\widecheck{\left(E/F_{n-1}\right)^*}\simeq\Ker\pi\simeq F_1$. Now we
know that $\Phi$
sends $F_{n-1}$ onto $(E/F_1)^*$ and $F_1$ onto $\Ker\pi$, so starting with
the following exact sequence:
\[
0 \ra \Ker\pi \hookrightarrow (E/F_1)^* \overset{\pi}{\ra} (F_{n-1}/F_1)^* \ra 0\]
we can obtain another by substituting $\Ker\pi$ with $F_1$ and $(E/F_1)^*$ with
$F_{n-1}$:
\[
0\ra F_1 \ra F_{n-1} \ra \widecheck{\left(F_{n-1}/F_1\right)^*} \ra 0.
\]
or in other terms $\widecheck{\left(F_{n-1}/F_1\right)^*} \simeq (F_{n-1}/F_1)$.
Note that the
isomorphism is given by $x\ra \Phi(x)|_{F_{n-1}}$ and is therefore
hermitian.

The proof may be summarized by the graph of interwoven exact sequences presented
in figure~\ref{fig:exact_sequences}.
\end{proof}

\begin{remark}
\label{good_lambda}
If $ae_\lambda = \lambda be_\lambda$ and $2\lambda \not\in \mathbb{N}$, then
$H(e_\lambda,e_\lambda)=0$. In fact $H(e_\lambda,e_\lambda)\in
E_0$ must satisfy:
\begin{multline*}
aH(e_\lambda,e_\lambda)=H(ae_\lambda,e_\lambda) + H(e_\lambda,-ae_\lambda) =\\
H(\lambda be_\lambda, e_\lambda) + H(e_\lambda, -\lambda be_\lambda) = 2\lambda
bH(e_\lambda,e_\lambda)
\end{multline*}
which has non-trivial solutions in $E_0$ only if $2\lambda\in\mathbb{N}$. The
double inversion of signs in the second
factor are due to the hermitian nature of the form.
\end{remark}

\begin{lemma}
\label{find_delta=0}
Let $E$ be a regular hermitian {\ab} with hermitian form $H$, $\lambda\in\CC$ and $j\in\NN$. If there exist two distinct
normal sub-{\ab}s $F\simeq E_{\lambda+j}$ and $G\simeq E_{\lambda}$ then there exist a normal sub-{\ab} $F_1$ such that
$H(F_1, F_1) = 0$.
\end{lemma}

\begin{proof}
Let $e_f$ and $e_g$ be generators of $F$ and $G$.

By the fundamental property $ab-ba=b^2$ of {\ab}s we have $ab^je_g = (\lambda + j)b\cdot b^je_g$. Consider now the
complex vector space:
\[
V := \{\alpha e_f + \beta b^je_g| \alpha, \beta\in\CC\}
\]
Note that every $v\in V$ satisfies $av = (\lambda + j)bv$. The properties of $H$
gives us:
\[
(a - 2{\lambda + j}b)H(v,w) = 0
\]
which has in $E_0$ only solutions of the form $\alpha b^{2(\lambda + j)}e_0$, $\alpha\in\CC$.  There exists therefore a
$\CC$-bilinear $B$ from $V\times V$ to $\CC$ such that:
\[
H(v,w) = B(v,w)b^{2(\lambda + j)}e_0\quad\forall v,w\in V
\]
which will have an isotropic vector $e$ such that $B(e,e) = 0$, i.e. $H(e, e) = 0$. By taking $F_1 = <e>$ we conclude.
\end{proof}

\begin{lemma}
Let $E$ be a regular {\ab} and:
\[
0\subsetneq \ldots F_{i-1} \subsetneq F_i \subsetneq F_{i+1} \subsetneq
\ldots E
\]
be a Jordan-Hölder composition series with $F_i/F_{i-1}\simeq E_{\lambda_i}$ for
all $i$ and suppose there is a $j$ such
that $\lambda_{j+1} \neq \lambda_j$ mod $\ZZ$.

Then we can find another Jordan-Hölder composition series that differs only in
the $j$-th term $F'_j$ such that
$F'_j/F_{j-1}\simeq E_{\lambda'_j}$ and $F_{j+1}/F'_j\simeq E_{\lambda'_{j+1}}$
with $\lambda_j = \lambda'_{j+1}$ mod
$\ZZ$ and $\lambda_{j+1} = \lambda'_j$ mod $\ZZ$, i.e. we can permute the
quotients up to an integer shift of the
parameters.
\end{lemma}

\begin{proof}
Let consider $G:=F_{j+1}/F_{j-1}$ and the canonical projection $\pi:E\ra
E/F_{j-1}$. $G$ is a rank two module. Using the classification of regular
{\ab}s of rank $2$ given by D.~Barlet in \cite{abmodules} we see that the
only two possibilities for $G$ are:
\[
G \simeq E_{\lambda_j} \oplus E_{\lambda_{j+1}}
\]
in which case we take $F_j'=\pi^{-1}(E_{\lambda_{j+1}})$ or 
\[
G \simeq E_{\lambda_{j+1}+1,\lambda_j}
\]
generated by $y$ and $t$ satisfying:
\begin{eqnarray*}
ay &=& \lambda_j by\\
at &=& \lambda_{j+1} bt + y
\end{eqnarray*}
that has also another set of generators: $t$ and $x:= y + (\lambda_{j+1} -
\lambda_j + 1)bt$ which satisfy:
\begin{eqnarray*}
ax &=& (\lambda_j + 1) bx\\
at &=& (\lambda_j - 1) bt + x.
\end{eqnarray*}
In this case we take $F'_j=\pi^{-1}(<x>)$.
\end{proof}

\begin{lemma}
\label{unique_F_1}
Let $E$ be a regular hermitian {\ab}. If there is a unique normal elementary sub-(a,b)-modules $F_1 \simeq E_\lambda$ with
$\lambda\in\ZZ$ (resp. $\lambda\in\ZZ+1/2$) and every composition series that starts with $F_1$ contains at least
another elementary quotient $E_\mu$ with $\mu$ integer (resp. integer+1/2).

Then there exists a normal sub-{\ab} $F_{n-1}$ of rank $n-1$ such that $\widecheck{\left(E/F_{n-1}\right)^*}\simeq F_1$
and $F_{n-1}/F_1$ is hermitian.
\end{lemma}

\begin{proof}
Let $F_1$ be the elementary sub-{\ab} of the hypothesis and $\{F_i\}$ a J-H sequence beginning with $F_1$ and such that
$E/F_{n-1}\simeq E_\mu$ with $\mu$ integer (resp.  integer+1/2). We can find such a sequence by using repeatedly the
previous lemma.

Consider the exact sequence:
\[
0 \ra F_{n-1} \ra E \ra (E/F_{n-1}) \ra 0
\]
and the adjoint sequence:
\[
0 \ra \widecheck{\left(E/F_{n-1}\right)^*} \overset{i}{\ra} \widecheck{E}^*
\overset{\pi}{\ra} \widecheck{F}_{n-1}^* \ra 0.
\]
The image of $i$ is a normal elementary sub-{\ab} of $\widecheck{E}^*$ isomorphic to $E_{-\mu}$. But there is only one
such sub-{\ab}: $\Phi(F_1)$. Thus $\widecheck{\left(F/F_{n-1}\right)}^* \simeq F_1$.

By replacing $\widecheck{E}^*$ by $E$ and $\widecheck{(E/F_{n-1})}^*$ by $F_1$ in
the sequence we obtain:
\[
0 \ra F_1 \overset{i}{\ra} E \ra \widecheck{F}_{n-1}^* \ra 0
\]
which is exact and $i$ is the inclusion of sub-{\ab}s, so $\widecheck{F}_{n-1}^*
\simeq (E/F_1)$ or equivalently $F_{n-1}
\simeq \widecheck{\left(E/F_1\right)^*}$. Note that the first isomorphism is
given by $\Phi^{-1}$, while the second by
the restriction of $\Phi$.

Consider the following sequence and its adjoint:
\begin{align*}
0 \ra F_{n-1}/F_1 & \ra  E/F_1  \ra E/F_{n-1} \ra 0\\
0 \ra \widecheck{\left(E/F_{n-1}\right)^*} & \ra
\widecheck{\left(E/F_1\right)^*}  \ra
\widecheck{\left(F_{n-1}/F_1\right)^*} \ra 0
\end{align*}
by replacing $\widecheck{\left(E/F_{n-1}\right)^*}$ and
$\widecheck{\left(E/F_1\right)^*}$ with $F_1$ and
$F_{n-1}$ we obtain:
\[
0 \ra F_1 \overset{\phi}{\ra} F_{n-1} \overset{\pi}{\ra}
\widecheck{\left(F_{n-1}/F_1\right)^*} \ra 0
\]
for the uniqueness of $F_1$, $\phi$ can only be (up to multiplication by a
complex number) the inclusion $F_1 \subset
F_{n-1}$ and hence $\widecheck{\left(F_{n-1}/F_1\right)^*} \simeq
(F_{n-1}/F_1)$. Note that $\pi$ is the restriction of
$\Phi$ to $F_{n-1}$, so the isomorphism is hermitian.
\end{proof}

We can now prove the theorem.

\begin{proof}[Proof of theorem \ref{big}]
We will prove the theorem by induction on the rank of the \ab. For rank $0$
and $1$ the theorem is obvious.

Suppose we proved the theorem for every rank $< n$ and let's prove it for rank $n$. Let find $F_1 \subset F_{n-1}$ of
rank $1$ and $n-1$ such that $\widecheck{\left(E/F_{n-1}\right)^*} \simeq F_1$ and $F_{n-1}/F_1$ is hermitian. We can
have different cases which are exhaustive:
\begin{enumerate}
\item We can find $G\simeq E_\lambda$, a normal elementary sub-{\ab} of $E$, with $2\lambda\not \in\ZZ$.
Then $\Phi(G)(G)=0$ by remark \ref{good_lambda} and we can apply lemma
\ref{delta=0}.

We still need to prove the induction step for {\ab}s whose only normal elementary sub-{\ab}s are isomorphic to
$E_\lambda$, with $2\lambda\in\ZZ$.

\item There are two distinct normal elementary sub-{\ab}s isomorphic to $E_n$ and $E_m$, where $m$, $n$ are integer or
half-integer. We apply lemma \ref{find_delta=0} and \ref{delta=0}.

The {\ab}s that were not included in the previous points have a unique normal
elementary sub-\ab isomorphic to $E_m$, where $m$ is integer or half-integer.

\item There is only one normal elementary sub-{\ab} of parameter equal to
$\lambda$ mod $\ZZ$, where $\lambda=0$ or
$1/2$, but at least two quotients of a J-H sequence are of parameter equal to
$\lambda$ mod $\ZZ$. We apply lemma
  \ref{unique_F_1}.

Only modules of rank at most $2$ (one for each possible value of $\lambda$)
still need to be checked.

\item The rank of $E$ is $2$ and one quotient of a J-H sequence is equal to $0$
mod $\ZZ$, the other equal to $1/2$ mod
$\ZZ$. By the classification of rank $2$ modules this case is impossible. In
fact with the notations of
  \cite{abmodules}:
\begin{eqnarray*}
\widecheck{\left(E_\lambda \oplus E_\mu\right)^*} &\simeq& E_{-\lambda} \oplus
E_{-\mu}\\
\widecheck{E}_{\lambda,\mu}^* &\simeq& E_{1-\lambda,1-\mu}\\
\end{eqnarray*}
so if $\lambda=0$ mod $\ZZ$ and $\mu=1/2$ mod $\ZZ$ the {\ab} is not
self-adjoint.
\end{enumerate}

By induction hypothesis $F_{n-1}/F_1$ has a J-H composition series that verifies
the theorem and by taking the inverse
image by the canonical morphism $F_{n-1} \ra F_{n-1}/F_1$ and adding $0$ and $E$
we find a J-H sequence of $E$ that
satisfies the theorem.
\end{proof}

Since for an anti-hermitian form $A$ we have $A(e,e)=0$ for every $e\in E$, by
using an anti-hermitian version of lemma
\ref{delta=0} alone and proceeding by induction, we can prove theorem~\ref{big}
in the anti-hermitian case.

We wish now to extend the result to all regular self-adjoint {\ab}s. We have
proven in \cite{karwasz} that every
regular {\ab} $E$ can be decomposed into a direct sum of hermitian or
anti-hermitian {\ab}s. We can hence prove the
following theorem:

\begin{theorem}
Let $E$ be a self-adjoint regular {\ab}. Then it admits a self-adjoint
Jordan-Hölder composition series.
\end{theorem}

\begin{proof}
  Let decompose $E$ into
  \[
  E = \bigoplus_{i=1}^m H_i
  \]
where $m$ is an integer, while the $H_i$ are either indecomposable self-adjoint
or of the form $G \oplus \widecheck{G}^*$,
  where $G$ is indecomposable non self-adjoint {\ab}.

Each term of this sum admits a self-adjoint composition series. In fact if $H_i$
is indecomposable self-adjoint, then
it is hermitian or anti-hermitian. We can therefore apply the previous
theorem~\ref{big}.
  
On the other hand if $H_i$ is the sum $G \oplus \widecheck{G}^*$ of a module and its
adjoint, we can easily find a
self-adjoint Jordan-Hölder composition series. Take in fact any Jordan-Hölder
series of $G$,
  \[
  0 = G_0 \subsetneq \dots \subsetneq G_n = G.
  \]
  and consider the adjoint series
  \[
0 = \widecheck{\left(G/G_n\right)}^* \subsetneq
\widecheck{\left(G/G_{n-1}\right)}^* \subsetneq
  \dots \widecheck{\left(G/G_0\right)}^* = \widecheck{G}^*.
  \]
Then the following composition series of $G \oplus \widecheck{G}^*$ is self-adjoint:  \begin{multline*}
0 = G_0 \subsetneq G_1 \subsetneq \dots G = G \oplus
\widecheck{\left(G/G_n\right)}^*
  \subsetneq G \oplus \widecheck{\left(G/G_{n-1}\right)}^* \subsetneq\\
  \dots \subsetneq G
  \oplus \widecheck{\left(G/G_0\right)}^* = G \oplus \widecheck{G}^*.
\end{multline*}
  
We will now prove the theorem on induction on $m$. The case $m=1$ was already
proven.

Suppose now $m\geq2$ and let $E':=H_1$ and~$F:=\sum_{i=2}^m H_i$. We have
therefore $E = E' \oplus F$, and $E'$ and
$F$ are both self-adjoint. By the remark above we can find a self-adjoint
composition series of $E'$:
  \[
  0 = E'_0 \subsetneq \dots \subsetneq E'_r = E'
  \]
while by induction hypothesis, we can find a self-adjoint composition series of
$F$:
  \[
  0 = F_0 \subsetneq \dots \subsetneq F_s = F.
  \]
  Then the following composition series is self-adjoint:
  \begin{multline*}
  0 = E'_0 \subsetneq E'_1 \subsetneq \dots \subsetneq E'_{[r/2]}
  \subsetneq E'_{[r/2]} \oplus F_1 \subsetneq \dots \subsetneq E'_{[r/2]}
  \oplus F_{[s/2]} [\cdots] \\
  E'_{[(r+1)/2]} \oplus F_{[(s+1)/2]} \subsetneq
  E'_{[(r+1)/2]} \oplus F_{[(s+1)/2] + 1} \subsetneq \dots \subsetneq
  E'_{[(r+1)/2]} \oplus F \\\
  \subsetneq E'_{[(r+1)/2] + 1} \oplus F \subsetneq \dots \subsetneq E'
  \oplus F,
  \end{multline*}
  where depending on the parity of $r$ and~$s$, $[\cdots]$ stands for
  \begin{enumerate}
    \item the $=$ sign if $r$ and $s$ are both even.
    \item the $\subsetneq$ sign if one is even and the other odd.
    \item the subsequence
      \[
      \subsetneq E'_{[r/2]} \oplus F_{[(s+1)/2]} \subsetneq
      \]
	
      This case needs a short verification. If $r$ and $s$ are odd, then
      the two central quotients of the series are isomorphic to
      $E'_{[(r+1)/2]}/E'_{[r/2]}$ and $F_{[(s+1)/2]}/F_{[s/2]}$. Since
      $E'_i$ and~$F_i$ are self-adjoint series both quotients are
      self-adjoint {\ab}s of rank~$1$. They are therefore isomorphic to
      $E_0$.
  \end{enumerate}
\end{proof}

\bibliographystyle{alpha}
\bibliography{biblio}

\begin{thebibliography}{{Kar}13}

\bibitem[{Bar}93]{abmodules}
Daniel {Barlet}.
\newblock {Theory of $(a,b)$-modules. I.}
\newblock In {\em {Complex analysis and geometry}}, pages 1--43. New York:
  Plenum Press, 1993.

\bibitem[{Bel}01]{belgrade}
R.~{Belgrade}.
\newblock {Dualit\'e et spectres des $(a,b)$-modules.}
\newblock {\em {J. Algebra}}, 245(1):193--224, 2001.

\bibitem[{Bri}70]{brieskorn}
Egbert {Brieskorn}.
\newblock {Die Monodromie der isolierten Singularit\"aten von Hyperfl\"achen.}
\newblock {\em {Manuscr. Math.}}, 2:103--161, 1970.

\bibitem[{Her}99]{hertlingbl}
Claus {Hertling}.
\newblock {Classifying spaces for polarized mixed Hodge structures and for
  Brieskorn lattices.}
\newblock {\em {Compos. Math.}}, 116(1):1--37, 1999.

\bibitem[{Kar}13]{karwasz}
Piotr~P. {Karwasz}.
\newblock {Hermitian $(a,b)$-modules and Saito's ``higher residue pairings''.}
\newblock {\em {Ann. Pol. Math.}}, 108(3):241--261, 2013.

\bibitem[{Sai}83]{saito}
Kyoji {Saito}.
\newblock {The higher residue pairings $K\sb F\sp{(k)}$ for a family of
  hypersurface singular points.}
\newblock {Singularities, Summer Inst., Arcata/Calif. 1981, Proc. Symp. Pure
  Math. 40, Part 2, 441-463 (1983).}, 1983.

\bibitem[{Sch}03]{schulze2}
Mathias {Schulze}.
\newblock {Monodromy of hypersurface singularities.}
\newblock {\em {Acta Appl. Math.}}, 75(1-3):3--13, 2003.

\bibitem[SS85]{steenbrink}
J.~{Scherk} and J.H.M. {Steenbrink}.
\newblock {On the mixed Hodge structure on the cohomology of the Milnor fibre.}
\newblock {\em {Math. Ann.}}, 271:641--665, 1985.

\bibitem[{Var}82]{varchenko}
A.N. {Varchenko}.
\newblock {Asymptotic Hodge structure in the vanishing cohomology.}
\newblock {\em {Math. USSR, Izv.}}, 18:469--512, 1982.

\end{thebibliography}
\end{document}